\documentclass[reqno,10pt]{amsart}
\usepackage{amsmath,amsthm,enumerate,graphics,amsfonts,latexsym,amsopn,verbatim,amscd,amssymb}
\usepackage{color}
\theoremstyle{plain}
\newtheorem{thm}{Theorem}[section]

\newtheorem{lem}[thm]{Lemma}
\newtheorem{cor}[thm]{Corollary}
\newtheorem{prop}[thm]{Proposition}

\newtheorem{problem}{Theorem}[section] 
\newtheorem{prob}[problem]{Problem}

\theoremstyle{definition}

\begin{document}
\title[Solvable graphs of  finite groups]{A note on solvable graphs of  finite groups}

\author[P. Bhowal, D. Nongsiang and R. K. Nath]
{P. Bhowal,    D. Nongsiang and R. K. Nath*}

\address{Parthajit Bhowal, Department of Mathematical Sciences, Tezpur University, Napaam-784028, Sonitpur, Assam, India.}

\email{bhowal.parthajit8@gmail.com}

\address{Deiborlang Nongsiang, Department of Mathematics, Union Christian College, Umiam-793122, Meghalaya, India.}

\email{ndeiborlang@yahoo.in}

\address{Rajat Kanti Nath, Department of Mathematical Sciences, Tezpur University, Napaam-784028, Sonitpur, Assam, India.}

\email{rajatkantinath@yahoo.com}


\begin{abstract}
Let $G$ be a finite non-solvable group with solvable radical $Sol(G)$.  The solvable graph $\Gamma_s(G)$ of $G$ is a graph  with vertex set $G\setminus Sol(G)$ and two distinct vertices $u$ and $v$ are adjacent if and only if $\langle u, v \rangle$ is solvable.  We show that $\Gamma_s (G)$ is not a star graph, a tree, an $n$-partite graph for any positive integer $n \geq 2$ and not a regular graph for any non-solvable finite group $G$. We compute the girth of $\Gamma_s (G)$ and derive a lower bound of the clique number of $\Gamma_s (G)$. We prove the non-existence of finite non-solvable groups whose solvable graphs are planar, toroidal, double-toroidal, triple-toroidal or projective. We conclude the paper by obtaining a relation between  $\Gamma_s (G)$ and the solvability degree of $G$.
\end{abstract}

\thanks{*Corresponding author}

\maketitle

\section{Introduction} \label{S:intro}
Let $G$ be a finite group and $u \in G$. The solvabilizer of $u$, denoted by $Sol_G(u)$, is the set given by $\{v \in G : \langle u, v \rangle \text{ is solvable}\}$. Note that the centralizer $C_G(u) := \{v \in G : uv = vu\}$  is a subset of $Sol_G(u)$ and hence the center $Z(G) \subseteq Sol_G(u)$ for all $u \in G$. By \cite[Proposition 2.13]{hr}, $|C_G(u)|$ divides $|Sol_G(u)|$ for all $u \in G$ though  $Sol_G(u)$ is not a subgroup of $G$ in general. A group $G$ is called a S-group if $Sol_G(u)$ is  a subgroup of $G$ for all $u \in G$. A finite group $G$ is a S-group if and only if it is solvable (see \cite[Proposition 2.22]{hr}). Many other properties of $Sol_G(u)$ can be found in \cite{hr}.  
We write
$Sol(G) = \{u \in G : \langle u, v \rangle \text{ is solvable for all } v \in G\}$. It is easy to see that $Sol(G) = \underset{u \in G}{\cap} Sol_G(u)$. Also, $Sol(G)$ is the solvable radical of $G$ (see \cite{gkps}).  The  solvable graph of a finite non-solvable group $G$ is a simple undirected graph whose vertex set is $G \setminus Sol(G)$,  and two vertices $u$ and $v$ are adjacent if $\langle u, v \rangle$ is a solvable. We write  $\Gamma_s (G)$ to denote this graph. It is worth mentioning that $\Gamma_s (G)$ is  the complement of the non-solvable graph of $G$ considered in \cite{hr, akbari18} and extension of commuting and nilpotent graphs of finite groups that are studied extensively in \cite{Ab06,aam2, AF14,amr06, bbhR09,darafsheh, das2, das2015,dutta-18,jdutta-17-1,jdutta-17-2,jdutta-18,Deibor-17,Talebi}. It is worth mentioning that the study of commuting graphs of finite groups is originated from  a question posed by Erd\"os \cite{neu}. 

  In this  paper, we show that $\Gamma_s (G)$ is not a star graph, a tree, an $n$-partite graph for any positive integer $n \geq 2$ and not a regular graph for any non-solvable finite group $G$. In Section 2, we also show that the girth of $\Gamma_s (G)$ is $3$ and the clique number of $\Gamma_s (G)$ is greater than or equal to $4$. In Section 3, we first show that for a given non-negative integer $k$, there are at the most finitely many finite non-solvable groups whose solvable graph have genus $k$. We also show that there is no finite non-solvable group, whose solvable graph is planar, toroidal, double-toroidal, triple-toroidal or projective. We conclude the paper by obtaining a relation between  $\Gamma_s (G)$ and $P_s(G)$ in Section 4, where $P_s(G)$ is the probability that a randomly chosen pair of elements of $G$ generate a solvable group (see \cite{gW2000}). 



The reader may refer to \cite{west} and \cite{whit} for various standard graph theoretic terminologies. For any subset $X$ of the vertex set of a graph $\Gamma$, we write $\Gamma[X]$ to denote the induced subgraph of $\Gamma$ on $X$. The  girth of    $\Gamma$ is the minimum of the lengths of all cycles in     $\Gamma$, and is denoted by ${\rm girth}(\Gamma)$.
We write $\omega (\Gamma)$ to denote the clique number of $\Gamma$ which is the least upper bound of the sizes of all the cliques of $\Gamma$.
The smallest non-negative integer $k$ is called the  genus of a graph $\Gamma$ if  $\Gamma$  can be embedded on the surface obtained by attaching $k$ handles to a sphere.  Let $\gamma(\Gamma)$ be the genus of $\Gamma$. Then, it is clear that $\gamma(\Gamma) \geq \gamma(\Gamma_0)$ for any subgraph $\Gamma_0$ of $\Gamma$.
Let $K_n$ be the complete graph on $n$ vertices and $mK_n$ the disjoint union of $m$ copies of $K_n$. 
It was proved in \cite[Corollary 1]{bhky} that $\gamma(\Gamma) \geq \gamma(K_m) + \gamma(K_n)$ if $\Gamma$ has two disjoint subgraphs isomorphic to $K_m$ and $K_n$. Also, by \cite[Theorem 6-38]{whit} we have
\begin{equation} \label{kn}
\gamma(K_{n})=\left\lceil \frac{(n-3)(n-4)}{12}\right
\rceil \text{ if } n \geq 3.
\end{equation}
\noindent A graph $\Gamma$ is called planar, toroidal, double-toroidal and triple-toroidal if $\gamma(\Gamma) = 0, 1, 2$ and $3$ respectively.

Let $N_k$ be the connected sum of $k$ projective planes. 
A simple graph which can be embedded in $N_k$ but not in $N_{k-1}$, is called a graph of crosscap $k$. The notation $\bar{\gamma}(\Gamma)$ stand for the crosscap of a graph $\Gamma$. It is easy to see that $\bar{\gamma}(\Gamma)\geq \bar{\gamma}(\Gamma_0)$ for any subgraph $\Gamma_0$ of $\Gamma$. It was shown in \cite{bou} that 
\begin{equation}\label{kbmn}
\bar{\gamma}(K_n)= \begin{cases} 
      \lceil\frac{1}{6}(n-3)(n-4)\rceil & \text{ if } n\geq 3 \text{ and } n\neq 7, \\
      3 & \text{ if } n = 7. \\
   \end{cases}
\end{equation}
\noindent A graph $\Gamma$ is called a  projective graph if $\bar{\gamma}(\Gamma) = 1$. It is worth mentioning that $2K_5$ is not projective graph (see \cite{ghw}).

\section{Graph realization}

We begin with the following lemma.

\begin{lem}\label{degree-lem}
For every $u\in G\setminus Sol(G)$ we have
\begin{center}
    $deg(u)=|Sol_G(u)|-|Sol(G)|-1$.
\end{center}
\end{lem}
\begin{proof}
Note that $deg(u)$ represents the number of vertices from $G\setminus Sol(G)$ which are adjacent to $u$. Since $u\in Sol_G(u)$, therefore $|Sol_G(u)|-1$ represents the number of vertices which are adjacent to $u$. Since we are excluding $Sol(G)$ from the vertex set therefore $deg(u)=|Sol_G(u)|-|Sol(G)|-1$.
\end{proof}
\begin{prop}\label{star}
$\Gamma_s(G)$ is not a star.
\end{prop}
\begin{proof}
Suppose for a contradiction $\Gamma_s(G)$ is a star. Let $|G|-|Sol(G)|=n$. Then there exists $u\in G\setminus Sol(G)$ such that $deg(u)=n-1$. Therefore, by Lemma \ref{degree-lem}, $|Sol_G(u)|=|G|$. This gives $u\in Sol(G)$, a contradiction. Hence, the result follows.
\end{proof}
\begin{prop} \label{comp}
$\Gamma_s(G)$ is not complete bipartite.
\end{prop}
\begin{proof}
Let $\Gamma_s(G)$ be complete bipartite.
Suppose that $A_1$ and $A_2$ are parts of the bi-partition.  Then, by Proposition \ref{star}, $|A_1|\ge 2$ and $|A_2|\ge 2$.
Let $u\in A_1, v \in A_2$.  If $|\langle u, v \rangle Sol(G)\setminus Sol(G)| > 2$, then there exists $y \in \langle u, v \rangle Sol(G)\setminus Sol(G)$ with $u \ne y \ne v$ such that $\langle u, y\rangle$ and $\langle v,y \rangle$ are both soluble.  But then $y \not \in A_1$ and $y \not \in A_2$, a contradiction.

It follows that $|\langle u, v \rangle Sol(G)\setminus Sol(G)| = 2$. In particular, $Sol(G)=1$ and $\langle u,v \rangle$ is cyclic of order $3$ or $|Sol(G)|=2$ and $v= uz$ for $z$ an involution in $Sol(G)$.
Now the neighbours of $u \in A_1$ is just $u^2\in A_2$ or $uz$ in the respective cases.  Hence $|A_2|= |A_1|=1$, a contradiction. Hence, the result follows.
\end{proof}
Following similar arguments as in the proof of Proposition \ref{comp} we get the following result.
\begin{prop}
$\Gamma_s(G)$ is not complete $n$-partite.
\end{prop}

\begin{prop} For any finite non-solvable group $G$, $\Gamma_s(G)$ has no isolated vertex.
\end{prop}
\begin{proof}
Suppose $x$ is an isolated vertex of $\Gamma_s(G)$. Then $|Sol(G)|=1$; otherwise $x$ is adjacent to $xz$ for any $z\in Sol(G)\setminus \{1\}$. Thus it follows that $o(x)=2$; otherwise $x$ is adjacent to $x^2$.
Let $y \in G$. Then $\langle x, x^y\rangle$ is dihedral and so $x= x^y$ as $x$ is isolated. Hence $x \in Z(G)$ and so $x \in Z(G)\le Sol(G)$, a contradiction.
 Hence, $\Gamma_s(G)$ has no isolated vertex.

\end{proof}

The following lemma is useful in proving the next two results as well as some results in  subsequent sections. 
\begin{lem} \label{lem-x-x^2}
Let $G$ be a finite non-solvable group. Then there exist $x\in G$ such that $x, x^2\not\in Sol(G)$.
\end{lem}
\begin{proof} 
Suppose that for all $x\in G$, we have $x^2\in Sol(G)$. Therefore, $G/Sol(G)$ is elementary abelian and hence solvable. Also, $Sol(G)$ is solvable. It follows that $G$ is solvable, a contradiction. Hence, the result follows.
\end{proof}

\begin{thm} \label{no3cyc} 
Let $G$ be a finite non-solvable group. Then $girth(\Gamma_s (G))=3$.
\end{thm}

\begin{proof} Suppose for a contradiction that $\Gamma_s (G)$ has no 3-cycle. Let $x\in G$ such that $x,x^2\not\in Sol(G)$ (Lemma \ref {lem-x-x^2} guarantees the existence of such element). Suppose $|Sol(G)| \geq 2$. Let $z\in Sol(G), z\neq 1$, then $x,x^2$ and $xz$ form a 3-cycle, which is a contradiction. Thus $|Sol(G)|=1$. In this case, every element of $G$ has order 2 or 3; otherwise, $\{x, x^2, x^3 \}$ forms a 3-cycle in $\Gamma_s (G)$ for all $x\in G$ with $o(x) >3$.  Therefore,  $|G|= 2^m 3^n$ for some non-negative integers $m$ and $n$. 
By Burnside's Theorem, it follows that $G$ is solvable; a contradiction.  Hence, $girth(\Gamma_s(G))=3$.
\end{proof}

\begin{thm} \label{clique-no}
Let $G$ be a finite non-solvable group. Then $\omega(\Gamma _s(G))\geq 4$.
\end{thm}
\begin{proof}
Suppose for a contradiction that $G$ is a finite non-solvable group with $\omega(\Gamma _s(G))\leq 3$. Let $x\in G\setminus Sol(G)$ such that $x^2\not \in G$. 
Existence of such element is guaranteed by Lemma \ref{lem-x-x^2}. 
Suppose $|Sol(G)|\geq 2$. Let $z\in Sol(G), z\neq 1$, then $\{x,x^2,xz,x^2z\}$ is a clique which is a contradiction. Thus $|Sol(G)|=1$. In this case every element of $G\setminus Sol(G)$ has order $2,3$ or $4$ otherwise $\{x,x^2,x^3,x^4\}$ is a clique with $o(x)>4$, which is a contradiction. Therefore $|G|=2^m3^n$ where $m,n$ are non-negative integers.
Again, by Burnside's Theorem, it follows that $G$ is solvable; a contradiction.  This completes the proof.
\end{proof}
As a corollary to Theorem \ref{no3cyc} and Theorem \ref{clique-no}   we have the following corollary.

\begin{cor}
The solvable graph of a finite non-solvable group is not a tree.
\end{cor}

We conclude this section with the following result.
\begin{prop}
$\Gamma_s(G)$ is not regular.
\end{prop}
\begin{proof}
Follows from \cite[Corollary 3.17]{hr}, noting the fact that a graph is regular if and only if its complement is regular.
\end{proof}

\section{Genus  and diameter }
We begin this section with the following useful lemma.
\begin{lem} \label{sol} 
Let $G$ be a finite group and $H$ a solvable subgroup of $G$. Then $\langle H,Sol(G)\rangle$ is a solvable subgroup of $G$.
\end{lem}
\begin{proof}
Since $Sol(G)$ is normal we have $\langle H,Sol(G)\rangle = HSol(G)$. Now the proof follows from the fact that solvability is inherent by extension and quotient as
\[
\frac{HSol(G)}{Sol(G)} \cong \frac{H}{H \cap Sol(G)}.
\]
\end{proof}

\begin{prop}\label{bound} 
Let $G$ be a finite non-solvable group such that $\gamma(\Gamma_s (G)) = m$.
\begin{enumerate}
\item   If $S$ is a nonempty subset of $G \setminus Sol(G)$  
such that $\langle x,y \rangle$ is solvable for all $x,y \in S$, then $|S|\leq \left\lfloor\frac{7+\sqrt{1+48m}}{2}\right\rfloor$. 
\item   $|Sol(G)|\leq \frac{1}{t-1}\left\lfloor\frac{7+\sqrt{1+48m}}{2}\right\rfloor$, where $t = \max \{o(xSol(G))\mid xSol(G) \in G/Sol(G)\}$.  
\item   If $H$ is a solvable subgroup of $G$, then $|H| \leq \left\lfloor\frac{7+\sqrt{1+48m}}{2}\right\rfloor +|H \cap Sol(G)|$.
\end{enumerate}
 \end{prop}
\begin{proof}
We have $\Gamma_s (G)[S] \cong K_{|S|}$ and $\gamma(K_{|S|}) = \gamma(\Gamma_s (G)[S])\leq \gamma(\Gamma_s (G))$. Therefore, if $m = 0$ then $\gamma(K_{|S|}) = 0$. This gives $|S|\leq 4$, otherwise $K_{|S|}$ will have a subgraph $K_5$ having genus $1$. 
If $m > 0$ then, by Heawood's formula \cite[Theorem 6.3.25]{west}, we have 
$$
|S| = \omega(\Gamma_s (G)[S]) \leq \omega(\Gamma_s (G)) \leq \chi(\Gamma_s (G)) \leq \left\lfloor \frac{7+\sqrt{1+48m}}{2}  \right\rfloor
$$
where $\chi(\Gamma_s (G))$ is the chromatic number of $\Gamma_s (G)$.  Hence part (a) follows. 

Part (b) follows from Lemma \ref{sol} and part (a) considering  $S= \overset{t-1}{\underset{i=1}\bigsqcup}   y^i Sol(G)$, where $y \in G \setminus Sol(G)$ such that $o(ySol(G))=t$. 

Part (c) follows from part (a)   noting that $H = (H\setminus Sol(G))\cup(H\cap Sol(G))$.
\end{proof}


\begin{thm}\label{fingen}
Let $G$ be a finite non-solvable group. Then  $|G|$ is bounded above by a function of  $\gamma(\Gamma_s (G))$.
\end{thm}
\begin{proof} 
Let $\gamma(\Gamma_s (G)) = m$ and  $h_m = \left\lfloor\frac{7+\sqrt{1+48m}}{2}\right\rfloor$. By Lemma \ref{sol}, we have $\Gamma_s(G)[xSol(G)] \cong K_{|Sol(G)|}$, where $x \in G \setminus Sol(G)$. Therefore by Proposition \ref{bound}(a), $|Sol(G)| \leq h_m$. 

Let $P$ be a Sylow $p$-subgroup of $G$ for any prime $p$ dividing $|G|$ having order $p^n$ for some positive integer $n$. Then $P$ is a solvable. Therefore, by Proposition \ref{bound}(c), we have $|P| \leq h_m + |Sol(G)| \leq 2h_m$. Hence, $|G|< (2h_m)^{h_m}$ noting that the number of primes less than $2h_m$ is at most $h_m$. 
This completes the proof.
\end{proof}
As an immediate consequence of Theorem \ref{fingen} we have the following corollary.
\begin{cor}
Let $n$ be a non-negative integer. Then there are at the most finitely many finite non-solvable groups $G$  such that $\gamma(\Gamma_s (G)) = n$.
\end{cor}

The following two lemmas are essential in proving the main results of this section.
\begin{lem}\cite[Lemma 3.4]{nong}\label{group7}
 Let $G$ be a finite group. 
\begin{enumerate}
\item If $|G|=7m$ and  the Sylow $7$-subgroup is normal in $G$, then  $G$ has an abelian subgroup of order at least $14$ or $|G|\leq 42$.

\item If $|G|=9m$, where $3\nmid m$ and  the Sylow $3$-subgroup is normal in $G$, then  $G$ has an abelian subgroup of order at least $18$ or $|G|\leq 72$.
\end{enumerate}
\end{lem}

\begin{lem} \label{lem120}
If $G$ is a non-solvable group of order not exceeding $120$ then $\Gamma_s(G)$ has a subgraph isomorphic to $K_{11}$ and $\gamma(\Gamma_s(G))\geq 5$.
\end{lem}
\begin{proof}
If $G$ is a non-solvable group and $|G| \leq 120$ then $G$ is isomorphic to $A_5$, $A_5 \times {\mathbb{Z}}_2$, $S_5$ or $SL(2,5)$. Note that $|Sol(A_5)| =  |Sol(S_5)| = 1$ and
 $|Sol(A_5 \times {\mathbb{Z}}_2)| = |Sol(SL(2,5))| = 2$. Also, $A_5$ has a solvable subgroup of order  $12$ and $S_5$, $A_5 \times {\mathbb{Z}}_2$, $SL(2,5)$ have solvable subgroups of order  $24$. 
It follows that $\Gamma_s(G)$ has a subgraph isomorphic to $K_{11}$. Therefore, by  \eqref{kn}, $\gamma(\Gamma_s(G)) \geq \gamma(K_{11})= 5$.
\end{proof}
\begin{thm}\label{planar}
The solvable graph of a finite non-solvable group is neither planar, toroidal, double-toroidal nor triple-toroidal. 
\end{thm}
\begin{proof}

Let $G$ be a finite non-solvable group. Note that it is enough to show  $\gamma(\Gamma_s(G))\geq 4$ to complete the proof. Suppose that $\gamma(\Gamma_s(G))\leq 3$. Let $x \in G \setminus Sol(G)$ such that $x^2 \not \in Sol(G)$. Such element exists by Lemma \ref{lem-x-x^2}. Since any two elements of the set $A = xSol(G) \cup x^2Sol(G)$ generate a solvable group, by Proposition \ref{bound}(a),  we have $2|Sol(G)| = |A| \leq \left\lfloor\frac{7+\sqrt{1+48\cdot 3}}{2}\right\rfloor = 9$. Thus $|Sol(G)|\leq 4$. Let $p$ be a prime divisor of $|G|$ and $P$ is a Sylow $p$-subgroup of $G$. Since $P$ is solvable, by Proposition \ref{bound}(c), we get $|P| \leq 9 + |P \cap Sol(G)| \leq 13$. If $|P| = 11$ or $13$ then $|P \cap Sol(G)| = 1$. Therefore, $\Gamma_s(G)[P \setminus Sol(G)] \cong K_{10}$ or $K_{12}$. Using \eqref{kn}, we get $\gamma(\Gamma_s(G)[P \setminus Sol(G)]) = 4$  or $6$. Therefore,  $\gamma(\Gamma_s(G)) \geq \gamma(\Gamma_s(G)[P \setminus Sol(G)]) \geq 4$, a contradiction. Thus $|P| \leq 9$ and hence $p \leq 7$. This shows that $|G|$ divides  $2^3.3^2.5.7$.

We consider the following cases.

\noindent \textbf{Case 1.}  $|Sol(G)|=4$.

If $H$ is a Sylow $p$-subgroup   of $G$ where $p =  5$ or $7$ then    $\langle H, Sol(G)\rangle$ is solvable since   $H$ is solvable (by Lemma \ref{sol}). We have $|H \cap Sol(G)| = 1$   and $|\langle H, Sol(G)\rangle| =  20, 28$ according as $p =  5, 7$ respectively.  Therefore $\Gamma_s(G)[\langle H,Sol(G)\rangle \setminus Sol(G)]\cong K_{16}$ or $K_{24}$. By \eqref{kn} we get $\gamma(\Gamma_s(G))\geq \gamma(\Gamma_s(G)[\langle H,Sol(G)\rangle \setminus Sol(G)]) \geq 13$, which is a contradiction.


Thus $|G|$ is a divisor of $72$. Therefore, by Lemma \ref{lem120} we have $\gamma(\Gamma_s(G)) \geq 5$, a contradiction.  

\noindent \textbf{Case 2.} $|Sol(G)|=3$. 

If $H$ is a Sylow $p$-subgroup   of $G$ where $p =  5$ or $7$ then    $\langle H, Sol(G)\rangle$ is solvable since   $H$ is solvable (by Lemma \ref{sol}). We have $|H \cap Sol(G)| = 1$   and $|\langle H, Sol(G)\rangle| =  15, 21$ according as $p =  5, 7$ respectively.  Therefore $\Gamma_s(G)[\langle H,Sol(G)\rangle \setminus Sol(G)]\cong K_{12}$ or $K_{18}$. By \eqref{kn} we get $\gamma(\Gamma_s(G))\geq \gamma(\Gamma_s(G)[\langle H,Sol(G)\rangle \setminus Sol(G)]) \geq 6$, which is a contradiction.


Thus $|G|$ is a divisor of $72$. Therefore, by Lemma \ref{lem120} we have $\gamma(\Gamma_s(G)) \geq 5$, a contradiction.  

\noindent \textbf{Case 3.} $|Sol(G)|=2$. 

If $H$ is a Sylow $7$-subgroup   of $G$   then    $\langle H, Sol(G)\rangle$ is solvable since   $H$ is solvable (by Lemma \ref{sol}). We have $|H \cap Sol(G)| = 1$   and $|\langle H, Sol(G)\rangle| =  14$.  Therefore $\Gamma_s(G)[\langle H,Sol(G)\rangle \setminus Sol(G)]\cong K_{12}$. By \eqref{kn} we get $\gamma(\Gamma_s(G))\geq \gamma(\Gamma_s(G)[\langle H,Sol(G)\rangle \setminus Sol(G)]) \geq 6$, which is a contradiction.
 Let $K$ be a Sylow $3$-subgroup of $G$. If  $|K| = 9$ then $\langle K, Sol(G)\rangle$ is solvable since   $K$ is solvable (by Lemma \ref{sol}). We have $|K \cap Sol(G)| = 1$   and $|\langle K, Sol(G)\rangle| =  18$.  Therefore $\Gamma_s(G)[\langle K, Sol(G)\rangle \setminus Sol(G)]\cong K_{16}$. By \eqref{kn} we get $\gamma(\Gamma_s(G))\geq \gamma(\Gamma_s(G)[\langle K, Sol(G)\rangle \setminus Sol(G)]) = 13$, which is a contradiction.

Thus $|G|$ is a divisor of $120$. Therefore, by Lemma \ref{lem120} we have $\gamma(\Gamma_s(G)) \geq 5$, a contradiction. 

\noindent \textbf{Case 4.} $|Sol(G)|=1$.

In this case, first we shall show that $7\nmid |G|$. On the contrary, assume that $7\mid |G|$.
Let $n$ be the number of Sylow $7$-subgroup of $G$. Then $n \mid 2^3.3^2.5$ and $n \equiv 1 (\mod 7)$. If $n \ne 1$  then $n \geq 8$. Let $H_1,  \dots, H_8$  be eight distinct Sylow $7$-subgroup of $G$. Then the subgraph
induced $\Gamma_S(G)[H_i \setminus Sol(G)]$ for each $1 \leq i \leq 8$ will contribute $\gamma(\Gamma_S(G)[H_i \setminus Sol(G)]) = 1$ to the genus of $\Gamma_S(G)$. Thus $$
\gamma(\Gamma_S(G)) \geq \overset{8}{\underset{i = 1}{\sum}}\gamma(\Gamma_S(G)[H_i \setminus Sol(G)]) = 8,
$$ a contradiction. Therefore, Sylow $7$-subgroup of $G$ is unique and hence normal. 
Since we have started with a non-solvable group, by Lemma \ref{group7}, it follows that $G$ has an abelian subgroup of order atleast $14$. Therefore, by \eqref{kn} we have $\gamma(\Gamma_S(G)) \geq \gamma(K_{13})  = 8$, a contradiction. Hence, $|G|$ is a divisor of  $2^3.3^2.5$.

Now, we shall show that $9\nmid |G|$. Assume that, on the contrary, $9\mid |G|$. If Sylow $3$-subgroup of $G$ is not normal in $G$, then the number
of Sylow $3$-subgroup is greater than or equal to $4$. Let $H_1, H_2, H_3$ be three Sylow $3$-subgroup of $G$. Then the  induced subgraph $\Gamma_S(G)[H_1\setminus Sol(G)] \cong
K_8$ and so it contributes $\gamma(\Gamma_S(G)[H_1\setminus Sol(G)]) = 2$ to
the genus of $\Gamma_S(G)$. If $|H_1 \cap H_2| = 1$, then the induced subgraph $\Gamma_S(G)[H_2 \setminus Sol(G)] \cong
K_8$ and so it contributes $+2$ to the genus $\Gamma_S(G)$. Thus 
$$
\gamma(\Gamma_S(G)) \geq \gamma(\Gamma_S(G)[(H_1 \cup H_2)\setminus Sol(G)]) = 4
$$
which is a contradiction. So assume that $|H_1 \cap H_2| = 3$. Similarly $|H_1 \cap H_3| = 3$ and $|H_2 \cap H_3| = 3$. Let $M = H_2 \setminus H_1$. Then $|M| = 6$. Also note that if $L = H_1 \cup H_2$ and $K = H_3 \setminus L$, then $|K| \geq 4$. Also $H_1 \cap M = H_1 \cap K = M \cap K = \emptyset$.

If $|K| \geq 5$ then $H_1$ contribute $+2$ to genus of $\Gamma_S(G)$, $M$ and $K$ each
contribute $+1$ to genus of $\Gamma_S(G)$. Hence genus of $\Gamma_S(G)$ is greater than or equal to $4$, a contradiction.

Assume that $|K| = 4$. In this case $|M \cap H_3| = 2$. Let $x \in M \cap H_3$.
  Then $H_1$ contribute $+2$ to
genus of $\Gamma_S(G)$, $M \setminus \{x\}$ and $K \cup \{x\}$ each contribute $+1$ to genus of $\Gamma_S(G)$. Hence genus of $\Gamma_S(G)$ is greater than
or equal to $4$, a contradiction.

These show that the Sylow $3$-subgroup of $G$ is unique and hence  normal in $G$. Therefore, by Lemma \ref{group7} and Lemma \ref{lem120},  $G$ has an abelian subgroup $A$ of order at least $18$. Hence, 
$$
\gamma(\Gamma_S(G)) \geq \gamma(\Gamma_S(G)[A \setminus Sol(G)]) \geq \gamma(K_{17}) =  16
$$
which is a contradiction.
 
 It follows that $9 \nmid |G|$ and $G$ is a divisor of $120$. Therefore, by Lemma \ref{lem120} we get  $\gamma(\Gamma_S(G)) \geq 5$, a contradiction. Hence, $\gamma(\Gamma_s(G))\geq 4$ and the result follows.
\end{proof}
%

The above theorem gives that $\gamma(\Gamma_s(G))\geq 4$. Usually,  genera of  solvable graphs of finite non-solvable groups are very large. For example, if $G$ is the smallest non-solvable group $A_5$ then $\Gamma_s(G)$ has $59$ vertices and $571$ edges. Also  $\gamma(\Gamma_s(G))\geq 571/6 - 59/2 + 1 = 68$ (follows from \cite[Corollary 6--14]{whit}).    The following theorem shows that the crosscap number of the solvable graph of a finite non-solvable group is greater than $1$. 

\begin{prop} 
The solvable graph of a finite non-solvable group is not projective.
\end{prop}
\begin{proof}
Suppose $G$ is a finite non-solvable group whose solvable graph is projective.  Note that if $\Gamma_s(G)$ has a subgraph isomorphic to $K_n$   then,  by   \eqref{kbmn}, we must have $n\leq 6$. Let $x\in G$, such that $x,x^2 \not \in Sol(G)$. Then $\Gamma_s(G)[xSol(G)\cup x^2Sol(G)]\cong K_{2|Sol(G)|}$. Therefore, $2|Sol(G)|\leq 6$ and hence $|Sol(G)|\leq 3$. 

Let $p\mid |G|$ be a prime and $P$ be a Sylow $p$-subgroup of $G$. Then $\Gamma_s(G)[P\setminus Sol(G)]\cong K_{|P\setminus Sol(G)|}$ since $P$ is solvable. Therefore, $|P\setminus Sol(G)| = |P| - |P\cap Sol(G)| \leq 6$ and hence $|P|\leq 9$. This shows that $|G|$ is a divisor of  $2^3.3^2.5.7$. 

If $7\mid |G|$  then the Sylow $7$-subgroup of $G$ is unique and hence normal in $G$; otherwise, let $H$ and $K$ be two Sylow $7$-subgroup of $G$. Then $|H \cap K| = |H \cap Sol(G)| = |K \cap Sol(G)| = 1$. Therefore, $\Gamma_s(G)[(H\cup K)\setminus Sol(G)]$ has a subgraph isomorphic to $2K_6$. Hence, $\Gamma_s(G)$ has a subgraph isomorphic to  $2K_5$, which is a contradiction. Similarly, if $9\mid |G|$, then the Sylow $3$-subgroup of $G$ is normal in $G$. Therefore, by Lemma \ref{group7}, it follows that $|G|\leq 72$ or $|G|$ is a divisor of $2^3.3.5$. In both the cases, by Lemma \ref{lem120}, $\Gamma_s(G)$  has  complete subgraphs isomorphic to $K_{11}$, which is a contradiction. This completes the proof.
\end{proof}


We conclude this section, by an observation and a couple of problems regarding the diameter and connectedness of $\Gamma_s(G)$. 
 Using the following programme in GAP\cite{gap}, we see that the solvable graph  of the groups $A_5, S_5, A_5 \times {\mathbb{Z}}_2, SL(2,5), PSL(3,2)$ and $GL(2,4)$ are connected with diameter  $2$. The solvable graphs of $S_6$ and $A_6$ are connected with diameters greater than $2$.

\vspace{.5cm}

\begin{verbatim}
g:=PSL(3,2);
sol:=RadicalGroup(g);
L:=[];
gsol:=Difference(g,sol);
for x in gsol do
 AddSet(L,[x]);
 for y in Difference(gsol,L) do
  if IsSolvable(Subgroup(g,[x,y]))=true then
   break;
  fi;
  i:=0;
  for z in gsol do
   if IsSolvable(Subgroup(g,[x,z]))=true and IsSolvable(Subgroup(g,[z,y]))=true 
   then
    i:=1;
    break;
   fi;
  od;
  if i=0 then
   Print("Diameter>2");
   Print(x,"   ",y);
  fi;
 od;
od;
\end{verbatim}

\vspace{.5cm}

 In this connection, we have the following problems.
 
\begin{prob} Is  $\Gamma_s(G)$ connected for any finite non-solvable group $G$?
\end{prob}

\begin{prob} 
Is there any finite bound for the diameter of $\Gamma_s(G)$ when $\Gamma_s(G)$ is connected?
\end{prob}

\section{Relations with solvability degree}

The solvability degree of a finite group $G$ is defined by the following ratio
\begin{align*}
P_s(G) := \frac{|\{(u, v) \in G \times G : \langle u, v\rangle \text{ is solvable}\}|}{|G|^2}.
\end{align*}
Using the  solvability criterion (see \cite[Section 1]{Dolfi12}),  

``A finite group is solvable if and only if every pair of its elements generates a
solvable group"

\noindent for finite groups we have  $G$ is solvable if and only if its solvability degree is $1$. It was shown in \cite[Theorem A]{gW2000} that $P_s(G) \leq \frac{11}{30}$ for any finite non-solvable group $G$. In this section, we study a few properties of $P_s(G)$ and  derive a connection between $P_s(G)$ and $\Gamma_s(G)$ for finite non-solvable groups $G$.   
We begin with the following lemma.
\begin{lem}\label{formula_Ps(G)}
Let $G$ be a finite group. Then $P_s(G) = \frac{1}{|G|^2}\underset{u \in G}{\sum}|Sol_G(u)|$.
\end{lem}
\begin{proof}
Let $\mathcal{S} =  \{(u, v) \in G \times G : \langle u, v\rangle \text{ is solvable}\}$. Then 
\[
\mathcal{S} = \underset{u \in G}{\cup}(\{u\} \times \{v \in G : \langle u, v\rangle \text{ is solvable}\}) = \underset{u \in G}{\cup}(\{u\} \times Sol_G(u)).
\]
Therefore, $|\mathcal{S}| = \underset{u \in G}{\sum}|Sol_G(u)|$. Hence, the result follows. 
\end{proof}

\begin{cor}
$|G|P_s(G)$ is an integer for any finite group $G$.
\end{cor}

\begin{proof}
By Proposition 2.16 of \cite{hr} we have that $|G|$ divides $\underset{u \in G}{\sum}|Sol_G(u)|$. Hence, the result follows from Lemma \ref{formula_Ps(G)}.
\end{proof}
We have the following lower bound for $P_s(G)$.
\begin{thm}\label{bound-lower}
For  any finite group $G$,
\[
P_s(G) \geq \frac{|Sol(G)|}{|G|} + \frac{2(|G| - |Sol(G)|)}{|G|^2}.
\]
\end{thm}

\begin{proof}
By Lemma \ref{formula_Ps(G)}, we have
\begin{align}\label{Ps(G)-bound1}
|G|^2P_s(G) &= \underset{u \in Sol(G)}{\sum}|Sol_G(u)| + \underset{u \in G \setminus Sol(G)}{\sum}|Sol_G(u)|\nonumber\\
&= |G||Sol(G)| + \underset{u \in G \setminus Sol(G)}{\sum}|Sol_G(u)|. 
\end{align}
By Proposition 2.13 of \cite{hr}, $|C_G(u)|$ is a divisor of $|Sol_G(u)|$ for all $u \in G$ where $C_G(u) = \{v \in G : uv = vu\}$, the centralizer of $u \in G$. Since $|C_G(u)| \geq 2$ for all $u \in G$ we have $|Sol_G(u)| \geq 2$ for all $u \in G$. Therefore
\[
\underset{u \in G \setminus Sol(G)}{\sum}|Sol_G(u)| \geq  2(|G| - |Sol(G)|).
\] 
Hence, the result follows from \eqref{Ps(G)-bound1}. 
\end{proof}
The following theorem shows that $P_s(G) > \Pr(G)$ for any finite non-solvable group where $\Pr(G)$ is the commuting probability of $G$ (see \cite{Gural06}).
\begin{thm}\label{bound-Pr}
Let $G$ be a finite group. Then $P_s(G) \geq \Pr(G)$ with equality if and only if $G$ is a solvable group.
\end{thm}

\begin{proof}
The result follows from Lemma \ref{formula_Ps(G)} and the fact that $\Pr(G) = \frac{1}{|G|^2}\underset{u \in G}{\sum}|C_G(u)|$ noting that $C_G(u) \subseteq Sol_G(u)$ and so $|Sol_G(u)| \geq |C_G(u)|$ for all $u \in G$.

The equality holds if and only if $C_G(u) = Sol_G(u)$ for all $u \in G$, that is $Sol_G(u)$ is a subgroup of $G$ for all $u \in G$. Hence, by Proposition 2.22 of \cite{hr}, the equality holds if and only if  $G$ is  solvable.
\end{proof}

Let $|E(\Gamma_s(G))|$ be the number of edges of the graph  the non-solvable graph $\Gamma_s(G)$ of $G$. The following theorem gives a relation between $P_s(G)$ and $|E(\Gamma_s(G))|$. 
\begin{thm}\label{connection}
Let $G$ be a finite non-solvable group. Then
\[
2|E(\Gamma_s(G))| = |G|^2P_s(G) + |Sol(G)|^2 + |Sol(G)| - |G|(2|Sol(G)| + 1).
\]
\end{thm}
\begin{proof}
We have 
\[
2|E(\Gamma_s(G))| = |\{(x, y) \in (G\setminus Sol(G)) \times (G\setminus Sol(G)) : \langle x, y\rangle \text{ is solvable}\}| - |G| + |Sol(G)|.
\]
Also
\begin{align*}
\mathcal{S} &= \{(x,y) \in G \times G : \langle x, y\rangle \text{ is solvable}\}\\
&= Sol(G) \times Sol(G)\quad \sqcup \quad Sol(G) \times (G \setminus Sol(G))\quad \sqcup \quad (G \setminus Sol(G)) \times Sol(G) \\
& \quad\sqcup \quad\{(x, y) \in (G\setminus Sol(G)) \times (G\setminus Sol(G)) : \langle x, y\rangle \text{ is solvable}\}.  
\end{align*}
Therefore
\begin{align*}
|\mathcal{S}|  &= |Sol(G)|^2 + 2|Sol(G)|(|G| - |Sol(G)|) + 2|E(\Gamma_s(G))| + |G| - |Sol(G)|\\
\Longrightarrow |G|^2P_s(G) &= |G|(2|Sol(G)| + 1) - |Sol(G)|^2 - |Sol(G)| + 2|E(\Gamma_s(G))|.  
\end{align*}
Hence, the result follows.
\end{proof}
We conclude this  paper noting that lower bounds for $|E(\Gamma_s(G))|$ can be obtained from Theorem \ref{connection} using the lower bounds given in Theorem \ref{bound-lower}, Theorem \ref{bound-Pr} and the lower bounds for $\Pr(G)$ obtained in \cite{nath2010}. 
%

\end{document}